\documentclass{amsart}

\usepackage{amssymb}
\usepackage{amscd}
\usepackage{graphicx}
\usepackage{amsmath}
\usepackage{setspace}
\usepackage[all]{xy}
\usepackage{color}
\usepackage{enumerate}

\newtheorem{theorem}{Theorem}[section]

\newtheorem{corollary}[theorem]{Corollary}

\newtheorem{lemma}[theorem]{Lemma}

\newtheorem{proposition}[theorem]{Proposition}

\theoremstyle{remark}

\newtheorem{remark}[theorem]{Remark}

\newcommand{\End}{\mathrm{End}}

\newcommand{\diag}{\mathrm{ diag}}
\newcommand{\Tr}{\mathrm{ Tr}}

\title[]{Decompositions of matrices by using commutators}

\author{Simion Breaz, Cristian Rafiliu}

\address{"Babe\c s-Bolyai" University, Faculty of Mathematics and Computer Science, Str. Mihail Kog\u alniceanu 1, 400084, Cluj-Napoca, Romania}

\email[Simion Breaz]{simion.breaz@ubbcluj.ro}

\email[Cristian Rafiliu]{cristianrafiliu2@gmail.com}

\begin{document}
	
	\begin{abstract}
		We will use commutators to provide decompositions of $3\times 3$ matrices as sums whose terms satisfy some polynomial identities, and we apply them to bounded linear operators and endomorphisms of free modules of infinite rank. In particular it is proved that every bounded operator of an infinite dimensional complex Hilbert space is a sum of four automorphisms of order $3$ and that every simple ring that is obtained as a quotient of the endomorphism ring of an infinitely dimensional vector space modulo its maximal ideal is a sum of three nilpotent subrings.
		
	\end{abstract}
	
	
	\subjclass[2010]{Primary: 15A24; Secondary: 16S50; 47B47.
	}
	
	\keywords{commutator; trace; nilpotent matrix; nilpotent subring; braceable truss}
	\maketitle
	
	\section{Introduction}
	
	If $R$ is a ring, an element $a\in R$ is a \textit{commutator} if there exist $x,y\in R$ such that $a=xy-yx$. The commutators are often used to provide information on the elements in some ring. 
	
	For instance, Harris proved in \cite{Har} that if every element of a division ring $D$ is a sum of commutators then all endomorphisms of a finite vectorial vector space over $D$ are sums of square-zero endomorphisms. Of course, such a result is not valid for matrices over fields, but it works for infinite dimensional vectors spaces, \cite{deSequins1}. In fact it is proved that for all quadratic splitting polynomials $p_1,\dots,p_4$ over the field, every such endomorphism $f$ is a sum of four endomorphisms $f=\alpha_1+\dots+\alpha_4$ such that $p_i(\alpha_i)=0$ for all $i=1,\dots,4$ (i.e., $f$ is a $p_1,\dots,p_4$-sum). This result was extended in \cite{Breaz} to all free modules of infinite rank by using the fact that every endomorphism of an infinite rank free module is a commutator, \cite{Mesyan}. We refer to \cite{Che}, \cite{Fac}, \cite{Her}, \cite{Mar} for other applications of commutators.
	Salwa, \cite[Theorem 4]{Salwa}, extended Harris' approach to division rings with surjective inner derivations. In particular he proved that if $D$ is a division ring then the matrix ring $D_3$, of all $3\times 3$ matrices over $D$, is a sum of three nilpotent subrings if and only if $D$ admits a surjective inner derivation $\mathrm{ad}_r:x\mapsto rx-xr$. In this way the author obtained an elegant extension of a result of Bokut, \cite{Bok}. We note that $3\times 3$ matrices are also used in \cite{Mar} to write commutators in $\mathbb{C}^\star$-algebras as linear combinations of projections. 
	
	Having as a starting point the proof of \cite[Theorem 4]{Salwa}, in this note we use the commutators to provide decompositions of $3\times 3$ matrices over general rings, Corollary \ref{cor-de-appl}. This result is then used to obtain the following results. For all $p_1,\dots,p_4\in\mathbb{C}[X]$ of degree $3$, every bounded linear operator on an infinite dimensional Hilbert space is a $p_1,\dots,p_4$-sum, Proposition \ref{prop:hilbert-dec}. In particular every bounded linear operator on an infinite dimensional Hilbert space is a sum of four nilpotent operators of nilpotency index $3$, or a sum of four automorphisms of order $3$. In the case of rings with a surjective inner derivation, we obtain that the matrix ring $R_3$ is a sum of three nil subrings of nilpotency index $3$. Moreover, in this case $R_3$ can be obtained by using three braceable trusses. In particular, these results are true for the endomorphism rings associated to infinite rank free modules or to some standard classes of simple rings, Corollary \ref{cor:inf-rank}. As a last application, we improve \cite[Corollary 7 (c)]{BC}, proving that if the trace of an $n\times n$ matrix $A$ is a sum of $k$-commutators then $A$ is a sum of $\lfloor\log_2(n)\rfloor + k + 2$ nilpotent matrices. 
	
	In this paper all rings are unital and associative. If $R$ is a ring then the ring of $n\times n$ matrices over $R$ is denoted by $R_n$. For a matrix $A=(a_{ij})\in R_n$, we will denote its diagonal by $\mathrm{diag}(A)=(a_{11},\dots,a_{nn})$. The trace of $A$ will be the element $\Tr(A)=a_{11}+\dots+a_{nn}$.
	
	\section{Preparatory lemmas}\label{sect-leme}
	
	Our approach is based on the following lemma, whose proof is elementary.
	
	\begin{lemma}\label{lema-baza}
		Let $R$ be a unital ring, and $a,b,c\in R$. There exist $x,p,q,r\in R$ such that   
		$$\begin{cases}
			xq+r=a \\
			p-q-r=b \\
			-p+q- qx=c
		\end{cases}
		$$
		if and only if $a+b+c$ is a commutator.
	\end{lemma}

	\begin{lemma}\label{lemma-matrici}
		Let $R$ be a unital ring. If $a,b,c,s,t,u\in R$ are elements such that $a+b+c-(s+t+u)$ is a commutator, then there exists an invertible matrix $U$ and an upper triangular matrix $T$ such that $\mathrm{diag}(T)=(s,t,u)$ and $\diag ( UTU^{-1})=(a,b,c)$. 
	\end{lemma}
	
	\begin{proof} For an arbitrary element $x\in R$, we consider the matrix
		$U(x)=\begin{pmatrix} x & 1 & 1\\ 1 & 1 & 0\\ 1 & 0 & 0\end{pmatrix}$. Let $T(p,q,r)=\begin{pmatrix} s & p & q\\ 0 & t & r\\ 0 & 0 & u\end{pmatrix}$. Then 
		$$\diag(U(x)T(p,q,r)U(x)^{-1})=(xq+r+u, p-q-r+t, -p+q- qx+s).$$ 
		From Lemma \ref{lema-baza} it follows that there exist $x_0, p_0, q_0, r_0\in R$ such that  
		$$\begin{cases}
			x_0q_0+r_0=a-u \\
			p_0-q_0-r_0=b-t \\
			-p_0+q_0-q_0x_0=c-s.
		\end{cases}$$
		Therefore, it is enough to consider $U=U(x_0)$ and $T=T(p_0,q_0,r_0)$. 
	\end{proof}
	
	\begin{corollary}\label{cor-de-appl}
		Let $R$ be a ring and $A\in R_3$ with $\diag(A)=(a,b,c)$. For $k = \overline{1,3}$, we consider the elements $s_k, t_k, u_k\in R$, and suppose that 
		$a + b + c - \sum_{k=1}^{3}(s_k + t_k + u_k)$ is a commutator. Then $A$ admits a decomposition 
		$$A=U\begin{pmatrix} s_1 & * & *\\ 0 & t_1 & *\\ 0 & 0 & u_1\end{pmatrix}U^{-1}+ \begin{pmatrix}s_2 & * & *\\ 0 & t_2 & *\\ 0 & 0 & u_2\end{pmatrix} + \begin{pmatrix}u_3 & 0 & 0\\ * & t_3 & 0\\ * & * & s_3\end{pmatrix}.$$ 
		
		Consequently, if $p_k = (X - s_k)(X - t_k)(X - u_k)$ and $s_1, t_1, u_1$ are from the center of $R$ then $A$ has a decomposition $A=A_1+A_2+A_3$ such that $p_k(A_k)=0$ for all $k\in \{1,2,3\}$.
	\end{corollary}
	
	\begin{proof}
		We know from Lemma \ref{lemma-matrici} that there exists an invertible matrix $U$ and an upper triangular matrix $T$ such that $\mathrm{diag}(T)=(s_1,t_1,u_1)$ and $\diag ( UTU^{-1})=(a-s_2-u_3,b-t_2-t_3,c-u_2-s_3)$. It follows that 
		$\diag(A-UTU^{-1})=(s_2+u_3,t_2+t_3,u_2+s_3)$, and it is easy to obtain the desired decomposition. 
	\end{proof}
	
	\section{Applications}
	
	\subsection{Infinite dimensional (complex) Hilbert spaces}
	
	As a first application, we obtain decompositions of bounded linear operators on complex Hilbert spaces. 
	
	\begin{proposition}\label{prop:hilbert-dec}
		Let $H$ be an infinite dimensional Hilbert space and denote by $B(H)$ the ring of all bounded linear operators on $H$. If $p_1,\dots,p_4\in \mathbb{C}[X]$ are polynomials of degree $3$ and $B\in B(H)$ then there exist $B_1,\dots,B_4\in B(H)$ such that 
		$$B=B_1+\dots +B_4, \text{ and } p_i(B_i)=0 \text{ for all } i\in\{1,\dots,4\}.$$   
	\end{proposition}
	
	\begin{proof}
		It is well known that all bounded linear operators on $H$ that are not of the form $\lambda 1_H+C$ with $\lambda\neq 0$ and $C$ compact are commutators, \cite{BP}. Therefore, all elements of $B(H)$ are sums of two commutators. We can work in the ring of $3\times 3$ matrices over $B(H)$, since $B(H)\cong B(H)_3$. 
		
		Let $A\in B(H)_3$ with $\diag(A)=(a,b,c)$. We denote by $\alpha_i,\beta_i,\gamma_i$ the roots of the polynomial $p_i$. Observe that we can write $A=M+N$ such that $\Tr(M)$ and $\Tr(N)$ are not of the form $\lambda 1_H+C$ with $C$ compact, so the traces of $M$ and $N$ are commutators (see also \cite{Hal}). Then we can apply Corollary \ref{cor-de-appl} to obtain decompositions of the form
		$$M=U\begin{pmatrix} \alpha_1 1_H & * & *\\ 0 & \beta_1 1_H & *\\ 0 & 0 & \gamma_1 1_H\end{pmatrix}U^{-1}+ \begin{pmatrix}\alpha_2 1_H & * & *\\ 0 & \beta_2 1_H & *\\ 0 & 0 & \gamma_2 1_H\end{pmatrix} + \begin{pmatrix}0 & 0 & 0\\ * & 0 & 0\\ * & * & 0\end{pmatrix}$$
		and 
		$$N=V\begin{pmatrix} \alpha_3 1_H& * & *\\ 0 & \beta_3 1_H& *\\ 0 & 0 & \gamma_3 1_H\end{pmatrix}V^{-1}+ \begin{pmatrix}0 & * & *\\ 0 & 0 & *\\ 0 & 0 & 0\end{pmatrix} + \begin{pmatrix}\alpha_4 1_H& 0 & 0\\ * & \beta_4 1_H & 0\\ * & * & \gamma_4 1_H\end{pmatrix}.$$ 
		Adding these two equalities we obtain the desired decomposition (since the elements $\alpha_i 1_H,\beta_i 1_H,\gamma_i 1_H\in B(H)$ are central).
	\end{proof}
	
	\begin{corollary}
		If $H$ is an infinite dimensional Hilbert space then every operator $A\in B(H)$ can be decomposed as a sum of at most four nilpotent operators of nilpotency index at most $3$, and as a sum of at most four automorphisms of order $3$. 
	\end{corollary}
	
	
	\subsection{Rings with a surjective inner derivation} 
	
	Recall that if $R$ is a ring then to every element $r\in R$ we can associate an \textit{inner derivation}, that is the map $\mathrm{ad}_r:R\to R$, $\mathrm{ad}_r(s)=rs-sr$. In the case when the ring $R$ has a surjective inner derivation we can improve the results from Section \ref{sect-leme}. We will apply these results to the endomorphism ring of a free module of infinite rank. Examples of division rings with surjective inner derivations are constructed in \cite{Cohn} and \cite{Salwa}.   
	
	\begin{lemma}\label{lem:ad-surj}
		Let $R$ be a ring, and suppose that there exists $r\in R$ such that $\mathrm{ad}_r$ is surjective. Then there exists an invertible matrix $U\in R_3$ with the following property: 
		
		\noindent For all  $a,b,c,s,t,u\in R$ there exists an upper triangular matrix $T\in R_3$ such that $\mathrm{diag}(T)=(s,t,u)$ and $\diag ( UTU^{-1})=(a,b,c)$.
	\end{lemma}
	
	\begin{proof}
		It is enough to adapt the proof of Lemma \ref{lemma-matrici}, taking $U=U(r)$.
	\end{proof}
	
	Using a similar approach as in the proof of Proposition \ref{prop:hilbert-dec}, we obtain some interesting decompositions of the matrix ring $R_3$. For the first property presented in the next result, recall that a subring $S$ of a ring is \textit{nilpotent} of \textit{nilpotency index} $k$ if $S^k=0$, but $S^{k-1}\neq 0$. 
	
	For the second, we will prove that the matrix ring $R_3$ can be constructed by using braces. Recall that the notion of brace was introduced in \cite{Rump} as an instrument that is useful in the study of non-degenerate solutions for set-theoretic Yang-Baxter equations. Brzezi\'nski proved in \cite{Brz} that every brace can be described by using trusses (see also \cite{Brz-R}). Recall that a \textit{truss} is a triple $(T,[-,-,-],\cdot)$, where  $[-,-,-]$ is a ternary operation on $T$ such that there exists an abelian group structure $(T,+)$ with $[x,y,z]=x-y+z$ (i.e. $(T,[-,-,-])$ is an abelian heap), the multiplicative operation $\cdot$ is associative, and it is distributive on $[-,-,-]$. We refer to \cite{Brz} for more details. A subtruss of $T$ is a subset $S\subseteq T$ that is closed under the operations $[-,-,-]$ and $\cdot$. If $(T,\cdot)$ is a group (whose unit is denoted by $1$), we say that the truss is \textit{braceable}. This means that if we consider the operation $x+_1 y=[x,1,y]$, then $(T,+_1,\cdot)$ is a brace, \cite[Corollary 3.10]{Brz}. Note that to every ring $(R,+,\cdot)$ we can associate a truss $(T(R),[-,-,-],\cdot)$, where $T(R)=R$ and $[x,y,z]=x-y+z$. If $X,Y,Z\subseteq T$, we will use the notation $[X,Y,Z]=\{[x,y,z]\mid x\in X, y\in Y,z\in Z\}$. 
	
	\begin{lemma}\label{lem:nilp-brace}
		Let $R$ be a ring. If $S$ is a nil subring of $R$ then $1+S$ is a braceable subtruss of $T(R)$. 
		
		If $S_1,S_2,S_3$ are nil subrings of $R$ such that $R=S_1+S_2+S_3$ then $R=[1+S_1,1+S_2,1+S_3]$.
	\end{lemma}
	
	\begin{proof}
		The first statement is obvious. 
		
		For the second, let $r\in R$. Then $-1+r=s_1+s_2+s_3$ with $s_i\in S_i$ for all $i=1,2,3$. It follows that $r=(1+s_1)-(1-s_2)+(1+s_3)\in[1+S_1,1+S_2,1+S_3]$.
	\end{proof}
	
	\begin{proposition}\label{prop:inner-surj}
		Let $R$ be a ring, and suppose that there exists $r\in R$ such that $\mathrm{ad}_r$ is surjective. 
		\begin{enumerate}[{\rm a)}]
			\item 
			The matrix ring $R_3$ is a sum of three nilpotent subrings of nilpotency index $3$.
			\item There exist three braceable subtrusses $K,L,M$ of the truss $T(R_3)$ such that $R_3=[K,L,M]$. 
		\end{enumerate}
	\end{proposition}
	
	\begin{proof}
		a) If we consider $s=t=u=0$ in Lemma \ref{lem:ad-surj}, it follows that for every $A\in R_3$ there exists an upper triangular matrix $T$ with zero diagonal such that the diagonal of $A-UTU^{-1}$ is zero. Then $A=T_u+T_l+UTU^{-1}$, where $T_u$ and $T$ are upper triangular matrices, $T_l$ is a lower triangular matrix, and all these have zero diagonals. If we denote by $S_u\subseteq R_3$ the subring of all upper triangular matrices with zero diagonal, and by $S_\ell\subseteq R_3$ the subring of all lower triangular matrices with zero diagonal, it follows that $R_3=S_u+S_\ell+US_uU^{-1}$.  
		
		b) This follows from a) and Lemma \ref{lem:nilp-brace}.
	\end{proof}

	We will apply the above result to endomorphism rings of free modules of infinite rank. Moreover, recall that the vector spaces of infinite rank can be used to construct simple rings. More precisely, if $V$ is an infinite dimensional vector space over a field $F$ then $I=\{f\in \End_F(V)\mid \dim_F(\mathrm{Im}(f))<\dim_F(V)\}$ is the maximal ideal of $\End_{F}(V)$. Therefore, the quotient ring $\End_F(V)/I$ is a simple ring.

	\begin{corollary}\label{cor:inf-rank}
		a) Let $F$ be a free $R$-module of infinite rank. Then the endomorphism ring $\End_R(F)$ is a sum of three nilpotent subrings of nilpotency index $3$. Consequently, there exist three braceable subtrusses $K,L,M$ of the truss $T(\End_R(F))$ such that $\End_R(F)=[K,L,M]$.
		
		b) Let $V$ and $I$ as above. Then the simple ring $\End_F(V)/I$ is a sum of three nilpotent subrings of nilpotency index $3$. Consequently, there exist three braceable subtrusses $K,L,M$ of the truss $T(\End_F(V)/I)$ such that $\End_F(V)/I=[K,L,M]$.
	\end{corollary}
	
	\begin{proof}
		It is enough to prove a). Since $F\cong F\oplus F\oplus F$, the endomorphism ring $\End_R(F)$ is isomorphic to the matrix ring $\End_R(F)_3$. It was proved in \cite[Proposition 12]{Mesyan} that the endomorphism ring of an free module of infinite rank has a surjective inner derivation. 
		Using Proposition \ref{prop:inner-surj}, we obtain the conclusion.  
	\end{proof}    
	
	\begin{remark}
		In \cite{Breaz} it is proved that every endomorphism of a free $R$-module $F$ of infinite rank can be decomposed as a sum of four endomorphisms that are annihilated by prescribed splitting polynomials of degree $2$ over the center of $R$. In general only three endomorphisms as before are not enough, \cite{deSequins2}. However, as in the proof of Corollary \ref{cor:inf-rank}, we can use Corollary \ref{cor-de-appl} to conclude that for all splitting polynomials $p_1,p_2,p_3$ of degree $3$ over the center of $R$, every endomorphism $\alpha$ of $F$ can be decomposed as $\alpha=\alpha_1+\alpha_2+\alpha_3$ such that $p_i(\alpha_i)=0$ for all $i=1,2,3$. 
	\end{remark}

	\subsection{Matrices whose traces are sums of commutators}
	
	If $A=(a_{ij})\in R_n$, we associate to $A$ its trace $\Tr(A)=\sum_{i=1}^{n}a_{ii}$.  
	It was proved in \cite{BC} that if $\Tr(A)$ is a sum of $k$-commutators then $A$ can be expressed as a sum of $n+2k+1$ nilpotent matrices. In the following we will improve this bound. We need some preliminary results.
	
	\begin{lemma}\label{lem:descend-b2} Let $n$ be a positive integer, and we denote $\left\lfloor\log_2 n\right\rfloor=m$. We consider the recursive sequence defined by $n_0 = n$ and $n_k = \lfloor\frac{n_{k-1}+1}{2}\rfloor$ for all integers $k\geq 1$. Then $$n_m=\begin{cases}
			1, & \text{ if }n=2^m;\\
			2, & \text{ if } n\neq 2^m.
		\end{cases}$$
		%
		%
		%
	\end{lemma}
	\begin{proof} 
		Suppose that $n=2^m + c$ with $c\in [1,2^m)$, and we will prove by induction that for all integers $0\leq k< m$ we have $n_{k} = 2^{m-k} + c_k$ for some $c_k\in [1,2^{m-k}]$. The verification step is obviously valid from our hypothesis. For the induction step, assume that $n_{k-1} = 2^{m-k+1} + c_{k-1}$ with $c_{k-1}\in [1,2^{m-k+1}]$. Using $k\leq m$, we obtain
		\begin{align*}n_k & =  \left\lfloor\frac{n_{k-1}+1}{2}\right\rfloor = \left\lfloor\frac{2^{m-k+1}+c_{k-1} + 1}{2}\right\rfloor\\ & =   2^{m-k} + \left\lfloor\frac{c_{k-1} + 1}{2}\right\rfloor = 2^{m-k} + c_k,\end{align*}
		where $c_k = \left\lfloor\frac{c_{k-1} + 1}{2}\right\rfloor\in [1, 2^{m-k}]$.
		
		Therefore, $n_{m-1} = 2 + c_{m-1}$, where $c_{m-1}\in [1,2]$, and we obtain $n_{m-1}\in \{3,4\}$, hence $n_m=2$.
		
		If $n=2^m$, it follows by a direct computation that $n_{m}=1$. 
	\end{proof}
	
	

If the following we will denote by $\mathbf{0}_k$ the  zero matrix from $R_k$, while $\mathbf{0}_{k,\ell}$ will denote the zero matrix with $k$ lines and $\ell$ columns. We will suppress the indices if their values can be deduced from context.  

\begin{lemma}\label{lem:redu-case-3}
	If $n\geq 2$ is an integer, and $A\in R_n$ then we can decompose $A$ as a sum 
	$$A=B+T_u+T_\ell+ \begin{pmatrix} \mathbf{0}_{n-v} & \mathbf{0}_{n-v,v}\\ \mathbf{0}_{v,n-v} & D \end{pmatrix},$$ where $B$ is a sum of $\lfloor\log_2(n)\rfloor$ nilpotent matrices with trace zero (of nilpotency index $2$), $T_u$ is an upper triangular matrix with zero diagonal, $T_\ell$ is a lower triangular matrix with zero diagonal, $D\in R_v$, and $v\in \{1,2\}$.
\end{lemma}

\begin{proof} If $A=(a_{ij})\in R_n$, we consider the matrix $N_n(A)$ defined as follows:
$$N_{n}(A) = \begin{pmatrix}
	a_{11}         & \dots  & 0       & 0    & 0       & \dots          & a_{11} \\
	0         & \dots  & 0       & 0    & 0       & \dots     & 0      \\
	\dots         & \dots  & \dots      & \dots   & \dots        & \dots       & \dots     \\
	0              & \dots  & a_{kk}  & 0    & a_{kk}  & \dots         & 0      \\
	0              & \dots  & 0       & 0    & 0       & \dots          & 0      \\
	0              & \dots  & -a_{kk} & 0    & -a_{kk} & \dots          & 0      \\
	\dots         & \dots  & \dots      & \dots   & \dots        & \dots       & \dots     \\
	0        & \dots  & 0       & 0    & 0       & \dots   & 0      \\
	-a_{11}        & \dots  & 0       & 0    & 0       & \dots          & -a_{11}\\
\end{pmatrix} \textrm{ for }n=2k+1,$$
respectively
$$N_{n}(A) = \begin{pmatrix}
	a_{11}  & \dots  & 0      & 0       & \dots   & a_{11} \\
	0       & \dots  & 0      & 0       & \dots   & 0      \\
	\dots   & \dots  & \dots  & \dots   & \dots   & \dots  \\
	0       & \dots  & a_{kk} & a_{kk}  & \dots   & 0      \\
	0       & \dots  & -a_{kk} & -a_{kk} & \dots  & 0      \\
	\dots   & \dots  & \dots   & \dots   & \dots  & \dots  \\
	0       & \dots  & 0      & 0       & \dots   & 0      \\
	-a_{11} & \dots  & 0      & 0       & \dots   & -a_{11}\\
\end{pmatrix} \textrm{ for }n=2k.$$

Then $A = N_n(A) + \begin{pmatrix}
	\mathbf{0}_{\left\lfloor\frac{n}{2}\right\rfloor} & \mathbf{0}\\ \mathbf{0} & D_{\left\lfloor\frac{n+1}{2}\right\rfloor}	\end{pmatrix} + B$, where $B$ is a matrix with zero diagonal, and $D_{\left\lfloor\frac{n+1}{2}\right\rfloor}\in R_{\left\lfloor\frac{n+1}{2}\right\rfloor}$. This decomposition provides for $n\in\{2,3\}$ the desired conclusion. 

{Assume that $n\geq 4$.} As in Lemma \ref{lem:descend-b2}, we define the sequence $(n_k)$ by
$ n_0 = n\geq 4$ and $n_k = \left\lfloor\frac{n_{k-1}+1}{2}\right\rfloor$. Inductively, we obtain the decompositions: 
\begin{align*}A =& N_{n_0}(A) + \begin{pmatrix} \mathbf{0}_{n-n_1} & \mathbf{0}\\ \mathbf{0} & A_1 \end{pmatrix} + B_1\\ 
	= & N_{n_0}(A) + \begin{pmatrix} \mathbf{0}_{n-n_1} & \mathbf{0}\\ \mathbf{0} & N_{n_1}(A_1) \end{pmatrix} + \begin{pmatrix} \mathbf{0}_{n-n_2} & \mathbf{0}\\ \mathbf{0} & A_2 \end{pmatrix} + B_2 = \dots \\ = & N_{n_0}(A) + \begin{pmatrix} \mathbf{0}_{n-n_1} & \mathbf{0}\\ \mathbf{0} & N_{n_1}(A_1) \end{pmatrix} + \begin{pmatrix} \mathbf{0}_{n-n_2} & \mathbf{0}\\ \mathbf{0} & N_{n_2}(A_2) \end{pmatrix} + \dots \\ 
	& + \begin{pmatrix} \mathbf{0}_{n-n_{[\log_2 n] - 2}} & \mathbf{0}\\ \mathbf{0} & N_{n_{[\log_2 n] - 2}}(A_{n_{[\log_2 n] - 2}}) \end{pmatrix} + \begin{pmatrix} \mathbf{0}_{n-v} & \mathbf{0}\\ \mathbf{0} & A_v \end{pmatrix} + B_v,\end{align*} where $v = n_{[\log_2 n] - 1}\in \{2, 3, 4\}$, and $B_1,\dots,B_v$ are matrices with zero diagonals. 

If $v\in\{2,3\}$ we apply one more step by using the decomposition presented in the beginning of the proof. If $v = 4$, then $\begin{pmatrix} \mathbf{0}_v & \mathbf{0}\\ \mathbf{0} & A_v \end{pmatrix} = \begin{pmatrix} \mathbf{0}_{n-2} & \mathbf{0}\\ \mathbf{0} & N_2(A_v) \end{pmatrix} + \begin{pmatrix} \mathbf{0}_{n-2} & \mathbf{0}\\ \mathbf{0} & A_{v+1} \end{pmatrix} + B_{v+1}$, where $B_{v+1}$ is a matrix with zero diagonal.

Therefore, we can decompose the matrix $A$ as a sum of $[\log_2 n]$ nilpotent matrices with zero traces, a matrix of the form $\begin{pmatrix} \mathbf{0}_{n-v} & \mathbf{0}\\ \mathbf{0} & D_v \end{pmatrix}$ with {$v\in \{1,2\}$}, and a matrix with zero diagonal.
%
%
\end{proof}

\begin{proposition}\label{prop:nilp-dec}
Suppose $n\geq 3$. If $A\in R_n$ and $\Tr(A)$ is a sum of $k$-commutators then $A$ is a sum of $\lfloor\log_2(n)\rfloor + k+2$ 
nilpotent matrices. 
\end{proposition} 

\begin{proof}
We will start with the case $n=3$ and $k=1$. In this case, we apply Lemma \ref{lemma-matrici} for $s=t=u=0$ to conclude that there is a nilpotent matrix $N$ such that 
$\diag(N)=\diag(A)$. It follows that $A=N+K+L$, where $K$ is an upper triangular matrix, $L$ is a lower triangular matrix, and $\diag(K)=\diag(L)=(0,0,0)$.

For $n=3$ and $k>1$ we write $A=A_1+\dots+A_k$ such that the traces of all matrices $A_i$ are commutators. We consider, for all $i=\overline{1,n}$ decompositions as above $A_i=N_i+K_i+L_i$ such that $N_i$ are nilpotent matrices, $K_i$ are upper triangular, and $L_i$ are lower triangular matrix, such that $\diag(K_i)=\diag(L_i)=(0,0,0)$. 
It follows that $A=N_1+\dots+N_k+K+L$, where $K=K_1+\dots+K_k$ and $L=L_1+\dots+L_k$ are triangular matrices with zero diagonals. 

Coming back to the general case, we apply Lemma \ref{lem:redu-case-3} to observe that we have a decomposition  $$A=B+T_u+T_l+ \begin{pmatrix} \mathbf{0}_{n-v} & \mathbf{0}\\ \mathbf{0} & D \end{pmatrix},$$ where $B$ is a sum of $\lfloor\log_2(n)\rfloor$ nilpotent matrices with trace zero (of nilpotency index $2$), $T_u$ is an upper triangular matrix with zero diagonal, $T_l$ is a lower triangular matrix with zero diagonal, and $v\in \{1,2\}$. Since $n\geq 3$, there exists a matrix $C\in R_3$ such that $$\begin{pmatrix} \mathbf{0}_{n-v} & \mathbf{0}\\ \mathbf{0} & D \end{pmatrix}=\begin{pmatrix} \mathbf{0}_{n-3} & \mathbf{0}\\ \mathbf{0} & C \end{pmatrix}.$$
We have $\Tr(A)=\Tr(C)$, so we have a decomposition $C=N_1+\dots+N_k+K+L$, where $N_1,\dots,N_k$ are nilpotent matrices (of nilpotency index $3$),  $K$ is an upper triangular matrix with zero diagonal, and $L$ is a lower triangular matrices with zero diagonal.

It follows that $A$ can be written 
\begin{align*}A=B& + \begin{pmatrix} \mathbf{0}_{n-3} & \mathbf{0}\\ \mathbf{0} & N_1 \end{pmatrix}+\dots +\begin{pmatrix} \mathbf{0}_{n-3} & \mathbf{0}\\ \mathbf{0} & N_k \end{pmatrix} \\ &+ \left(T_u+\begin{pmatrix} \mathbf{0}_{n-3} & \mathbf{0}\\ \mathbf{0} & K \end{pmatrix}\right)+\left(T_l+\begin{pmatrix} \mathbf{0}_{n-3} & \mathbf{0}\\ \mathbf{0} & L \end{pmatrix}\right),\end{align*}
a sum of $\lfloor\log_2(n)\rfloor+k+2$ nilpotent matrices.
\end{proof}

\begin{remark}
For the decomposition obtained in the proof of Theorem \ref{prop:nilp-dec} we have more details about the nilpotent matrices: $\lfloor\log_2(n)\rfloor$ are of nilpotency index $2$, and $k$ of them are of nilpotency index $3$. Using the proof of \cite[Theorem 1]{Breaz}, it follows that every $2\times 2$ matrix whose trace is a commutator can be decomposed as a sum of two nilpotent matrices of nilpotency index $2$ and a matrix with zero diagonal. We can addapt the proof of Proposition \ref{prop:nilp-dec} to conclude that every matrix $A\in R_n$ such that $\Tr(A)$ is a sum of $k$-commutators can be decomposed as a sum of $\lfloor\log_2(n)\rfloor + 2k$ square zero matrices, an upper triangular matrix with zero diagonal and a lower triangular matrix with zero diagonal. Since every triangular matrix from $R_n$ with zero diagonal is a sum of $(n-1)$ square zero matrices, it follows that every matrix $A$ with $\Tr(A)$ a sum of $k$ commutators can be decomposed as a sum of  $\lfloor\log_2(n)\rfloor + 2k+2(n-1)$ square-zero matrices.
\end{remark}

\section*{Acknowledgements}
The research of S.\ Breaz is supported by a grant of the Ministry of Research, Innovation and Digitization, CNCS/CCCDI--UEFISCDI, project number PN-III-P4-ID-PCE-2020-0454, within PNCDI III. 

The research of C. Rafiliu is supported by a Babe\c s-Bolyai University Young Fellowship.

\end{document}